\documentclass[12pt]{article}
\usepackage{amsmath}
\usepackage{amsmath,amssymb, amsthm, latexsym, amsfonts, epsfig, color,graphicx,}
\usepackage{mathrsfs}
\usepackage{indentfirst}
\usepackage{chngcntr}
\usepackage{bbm}
\usepackage{dsfont}
\usepackage{newtxmath} 
\allowdisplaybreaks
\usepackage[colorlinks=True,linkcolor=blue,anchorcolor=blue,citecolor=blue,CJKbookmarks=true]{hyperref}
\usepackage{geometry}
\begin{document}
	\renewcommand{\a}{\alpha}
	\newcommand{\D}{\Delta}
	\newcommand{\ddt}{\frac{d}{dt}}
	\counterwithin{equation}{section}
	\newcommand{\e}{\epsilon}
	\newcommand{\eps}{\varepsilon}
	\newtheorem{theorem}{Theorem}[section]
	\newtheorem{proposition}{Proposition}[section]
	\newtheorem{lemma}[theorem]{Lemma}
	\newtheorem{remark}[theorem]{Remark}
	\newtheorem{example}{Example}[section]
	\newtheorem{definition}{Definition}[section]
	\newtheorem{corollary}[theorem]{Corollary}
	\makeatletter
	\newcommand{\rmnum}[1]{\romannumeral #1}
	\newcommand{\Rmnum}[1]{\expandafter\@slowromancap\romannumeral #1@}
	\makeatother

		\title{\bf Unique continuation inequalities for the Dunkl–Schr\"odinger equation via uncertainty principles}
		\author{Xingyu Zhao, Hui Xu\footnote{Corresponding author.}, Zhiwen Duan \\
			{\small {\it  School of Mathematics and Statistics, Huazhong University of Science }}\\
			{\small {\it and Technology, Wuhan, {\rm 430074,} P.R.China}}  \\
			{\small {\it Email: zhaoxingyu@hust.edu.cn~(X.Zhao)}}\\
			{\small {\it Email: xh202180016@outlook.com~(H.Xu)}}\\
			{\small {\it Email: duanzhw@hust.edu.cn~(Z.Duan)}}
		}
		\date{}
		\maketitle
		\centerline{\large\bf Abstract }
	In this paper, we establish unique continuation inequalities at two time points for the Dunkl--Schr\"odinger equation. The proof is based on quantitative uncertainty principles for the Dunkl transform. In particular, we prove that pairs of $(\varepsilon,k)$-thin sets form strong annihilating pairs for the Dunkl transform, which yields quantitative unique continuation properties for solutions to the Dunkl--Schr\"odinger equation.

\section{Introduction}
The purpose of this paper is to establish quantitative unique continuation inequalities for the Dunkl--Schr\"odinger equation
\begin{equation}\label{DQ}
	\begin{cases}
		i\partial_tu(x,t)+\Delta_ku(x,t)=0,
		& (x,t)\in\mathbb{R}^d\times(0,\infty),\\
		u(x,0)=u_0(x),
		& u_0\in L_k^2(\mathbb{R}^d),
	\end{cases}
\end{equation}
where $L_k^2(\mathbb{R}^d)$ is the weighted space associated with the measure
$ d\mu_k(x)=h_k^2(x)\,dx $, and $h_k$ is the Dunkl weight. Here
\[
\Delta_k=\sum_{j=1}^d \mathcal{T}_j^2
\]
is the Dunkl Laplacian, with $\{\mathcal{T}_j\}_{j=1}^d$ denoting the Dunkl operators associated with a finite reflection group $G$ and a nonnegative multiplicity function $k$. Dunkl operators, introduced by Dunkl in \cite{CFD1}, are differential--difference operators naturally associated with finite reflection groups. They play an important role in both pure mathematics and mathematical physics, and provide a useful framework for the study of special functions related to root systems \cite{CFD2}. Moreover, Dunkl operators are closely connected with certain Schr\"odinger operators arising in Calogero--Sutherland type quantum many-body systems \cite{BF,H96,VV}. They also give rise to generalizations of several classical analytic structures, including the Laplace operator, Fourier transform, heat semigroup, wave equation, and Schr\"odinger equation \cite{MFE2,CFD3,HM1,HM,HM2,HM3,HM4,MR}. In particular, when $k=0$, the Dunkl--Schr\"odinger equation reduces to the classical free Schr\"odinger equation.

The operator $-\Delta_k$ is densely defined, symmetric, and positive on $L_k^2(\mathbb{R}^d)$. Therefore, by the Friedrichs extension theorem, it admits a positive self-adjoint extension, and the corresponding Cauchy problem is well posed (see \cite{AB}). Moreover, it is known that the solution of \eqref{DQ} admits the representation \cite{HM}
\begin{equation}\label{key}
	u(x,t)
	=
	\frac{1}{(2t)^{\gamma_k+\frac d2}}
	e^{-i(d+2\gamma_k)\frac{\pi}{4}\operatorname{sgn}t}
	e^{i|x|^2/4t}
	\left[
	D_k\!\left(e^{i|\cdot|^2/4t}u_0\right)
	\right]\!\left(\frac{x}{2t}\right),
\end{equation}
where $D_k$ denotes the Dunkl transform and
\(\gamma_k=\sum_{\alpha\in R_+}k_\alpha
\) (see Section \ref{sec2} for the precise definitions).
This explicit formula shows that the behavior of solutions to \eqref{DQ} is closely linked to the Dunkl transform, and therefore quantitative questions for \eqref{DQ} are naturally related to uncertainty principles for the Dunkl transform.

The study of observability and quantitative unique continuation properties for Schr\"odinger equations has attracted considerable attention in recent years. Uncertainty principles have proved to be an effective tool in this direction. For the free Schr\"odinger equation, the Logvinenko--Sereda theorem \cite{Kov} and Nazarov's uncertainty principle \cite{Jaming} were used in \cite{HWW,WWZ} to establish observability inequalities from measurable sets and at two different time points. Similar unique continuation inequalities for linear Schr\"odinger equations with decaying potentials were obtained in \cite{HS},  while corresponding inequalities for nonlinear Schr\"odinger equations were established in \cite{WLH}.
More recently, these ideas have been extended to other transform settings. In \cite{WDX,XWD}, uncertainty principles associated with the Hankel transform were applied to Schr\"odinger equations with inverse-square potentials, while \cite{Xu} studied observability properties for Bessel-type Schr\"odinger equations via the canonical Fourier--Bessel transform. In the Dunkl setting, quantitative unique continuation inequalities for Dunkl--Schr\"odinger and Dunkl--Hermite equations were established in \cite{Muk}, where the corresponding estimates were obtained for sets of finite measure.

Motivated by these works, we establish a quantitative uncertainty principle for the Dunkl transform in terms of thin sets. Following the notion of thin sets  introduced in the Fourier transform setting in \cite{SVW}, we introduce the following definition in the Dunkl setting.
\begin{definition}
	Let $0<\varepsilon<1$ and define
	\[
	\rho(x)=\min(1,|x|^{-1}),
	\qquad x\in\mathbb{R}^d.
	\]
	A measurable set $S\subset\mathbb{R}^d$ is said to be $(\varepsilon,k)$-thin if
	\[
	\mu_k\bigl(S\cap B(x,\rho(x))\bigr)
	\le
	\varepsilon\,\mu_k\bigl(B(x,\rho(x))\bigr),
	\qquad x\in\mathbb{R}^d.
	\]
\end{definition}
The main objective of this paper is to show that pairs of $(\varepsilon,k)$--thin sets form strong annihilating pairs for the Dunkl transform. More precisely, we establish the following quantitative uncertainty inequality:
\begin{equation}\label{UP}
	\|f\|_{L_k^2}^2
	\le
	C\Bigl(
	\|f\|_{L_k^2(S^c)}^2
	+
	\|D_k f\|_{L_k^2(\Sigma^c)}^2
	\Bigr),
\end{equation}
where $S$ and $\Sigma$ are $(\varepsilon,k)$--thin sets. A key difficulty in the Dunkl setting, in contrast to the classical Fourier or Hankel transform cases, is that the generalized translation operator does not admit an explicit kernel representation. As a consequence, standard pointwise decay estimates are not directly available.
To address this issue, we build on Theorem 4.1 in \cite{AJ} and extend its framework by removing the smoothness assumptions on the functions. In particular, we obtain decay estimates for the generalized Dunkl translation of dilated compactly supported functions (see Lemma \ref{cutoff} and Remark \ref{re1}).
This refined estimate plays a crucial role in the proof of the strong annihilating pair property for the Dunkl transform.


Applying this uncertainty principle, we then derive the following unique continuation inequality at two time points for the Dunkl--Schr\"odinger equation \eqref{DQ}, which is the main result of this paper. Here and in the sequel, $u(x,t;u_0)$ denotes the solution to \eqref{DQ} corresponding to the initial datum $u_0$.
\begin{theorem}\label{T1}
	Assume that $\varepsilon$ is sufficiently small. Let $A,B\subset\mathbb{R}^d$ be $(\varepsilon,k)$-thin sets and let $T>S\ge0$. Then there exists a constant $C=C(k,d,A,B)>0$ such that, for every $u_0\in L_k^2(\mathbb{R}^d)$,
	\begin{equation}\label{T1.1}
		\int_{\mathbb{R}^d}|u_0(x)|^2\,d\mu_k(x)
		\le
		C\left(
		\int_{A^c}|u(x,S;u_0)|^2\,d\mu_k(x)
		+
		\int_{(2(T-S)B)^c}|u(x,T;u_0)|^2\,d\mu_k(x)
		\right).
	\end{equation}
(Here and in what follows, \(C(\cdot)\) stands for a positive constant depending only on the quantities enclosed in the brackets, which may vary in different contexts.)
\end{theorem}
The proof of Theorem \ref{T1} combines the uncertainty principle \eqref{UP} with the explicit representation \eqref{key} and a standard argument linking observability at two time points to the uncertainty principle.

Several remarks on Theorem \ref{T1} are in order:\\
($\mathbf{a_1}$) Inequality \eqref{T1.1} can be understood from two different perspectives. 
From the unique continuation perspective, \eqref{T1.1} is a unique continuation inequality at two time points for the Dunkl--Schr\"odinger equation \eqref{DQ}. In particular, if
\[
u(x,S;u_0)=0 \quad \text{on } A^c,
\qquad
u(x,T;u_0)=0 \quad \text{on } \bigl(2(T-S)B\bigr)^c,
\]
then it follows that
\[
u(x,t;u_0)= 0
\quad \text{on } \mathbb{R}^d\times [0,\infty).
\]
From the observability perspective, \eqref{T1.1} is an observability inequality at two time points. It states that, when $A$ and $B$ are $(\varepsilon,k)$-thin sets, the initial datum can be recovered from observations of the solution outside $A$ at time $S$ and outside the dilated set $2(T-S)B$ at time $T$. \\
($\mathbf{a_2}$) When $k=0$, inequality \eqref{T1.1} reduces to a two-time unique continuation inequality for the free Schr\"odinger equation.  \\
\textbf{Plan of the paper.} The rest of the paper is organized as follows: In Section \ref{sec2}, we introduce the notation and recall the necessary background on Dunkl operators, the Dunkl transform, generalized translations, and Dunkl convolutions. In Section \ref{sec3}, we establish the strong annihilating pair property for the Dunkl transform. Finally, in Section \ref{sec4}, we apply this result to derive a unique continuation inequality for the Dunkl--Schr\"odinger equation.
\section{Preliminaries}\label{sec2}
In this section, we introduce the notation used throughout the paper and review several standard facts concerning Dunkl operators, the Dunkl transform, generalized translation, and Dunkl convolution. For more details on Dunkl theory, we refer the reader to \cite{MFE2,MR3,ST}.
\subsection{Notation}
\begin{itemize}
	\item $\langle \cdot, \cdot \rangle$ denotes the standard Euclidean scalar product in $\mathbb{R}^d$. For $x \in \mathbb{R}^d$, we write $|x| = \sqrt{\langle x, x \rangle}$.
    \item $B(x,r)$ denotes the closed ball in $\mathbb{R}^d$ centered at $x$ with radius $r$.
	\item The weighted measure is $d\mu_k(x) := h_k^2(x) dx$, where $h_k$ is defined below.
\item For a measurable set $A \subset \mathbb{R}^d$, $\chi_A$ denotes the characteristic function of $A$, $A^c$ denotes the complement of $A$, and
\[
\mu_k(A):=\int_A d\mu_k(x).
\]
For each $\lambda>0$, we define
\[
\lambda A:=\{\lambda x:x\in A\}.
\]
	\item $L^p_k(\mathbb{R}^d)$ ($1 \le p \le \infty$) is the space of measurable functions on $\mathbb{R}^d$ with
	\[
	\|f\|_{L^p_k} := \left( \int_{\mathbb{R}^d} |f(x)|^p d\mu_k(x) \right)^{1/p} < \infty \quad (1 \le p < \infty),
	\]
	and $\|f\|_{L^\infty_k} := \operatorname{ess\,sup}_{x\in\mathbb{R}^d} |f(x)|$.
	\item $C_0(\mathbb{R}^d)$ denotes the space of continuous functions vanishing at infinity.
	\item For $1\le p\le \infty$, $L_{\mathrm{rad}}^p(\mathbb{R}^d,d\mu_k)$ denotes the subspace of radial functions in $L_k^p(\mathbb{R}^d)$.
	\item $O(N;\mathbb{R})$ denotes the real orthogonal group of $N \times N$ matrices satisfying $A^{\mathsf{T}}A = I_N$.
\end{itemize}

\subsection{Dunkl theory}
Let $G$ be a finite real reflection group on $\mathbb{R}^d$ with a fixed positive root system $R_+$, normalized such that $\langle \alpha, \alpha \rangle = 2$ for all $\alpha \in R_+$. For a nonzero $\alpha \in \mathbb{R}^d$, let $\sigma_\alpha$ be the reflection in the hyperplane orthogonal to $\alpha$:
\[
x\sigma_\alpha := x - 2\frac{\langle x,\alpha \rangle}{|\alpha|^2}\alpha, \qquad x\in\mathbb{R}^d.
\]
Let $k: R_+ \to [0,\infty)$ be a $G$-invariant multiplicity function, i.e. $k_\alpha = k_\beta$ whenever $\sigma_\alpha$ and $\sigma_\beta$ are conjugate in $G$. Define $\gamma_k := \sum_{\alpha\in R_+} k_\alpha$. 
The associated weight function is
\[
h_k(x) := \prod_{\alpha\in R_+} |\langle x,\alpha\rangle|^{k_\alpha}, \qquad x\in\mathbb{R}^d.
\]
It is known that there exists a constant $C>0$ such that for all $x\in\mathbb{R}^d$ and $r>0$,
\begin{equation}\label{vs}
C^{-1}\mu_k(B(x,r))
\le
 r^d \prod_{\alpha\in R_+}\bigl(|\langle x,\alpha\rangle|+r\bigr)^{2k_\alpha} \leq C\mu_k(B(x,r)).
\end{equation}

The Dunkl operators are defined as
\[
\mathcal{T}_i f(x) = \partial_i f(x) +  \sum_{\alpha \in R_+} k_{\alpha} \frac{f(x) - f(x\sigma_{\alpha})}{\langle x, \alpha \rangle} \langle \alpha, \epsilon_i \rangle, \quad 1 \leq i \leq d,
\]
where 
\(\epsilon_1, \dots, \epsilon_d\) are the standard unit vectors of \(\mathbb{R}^d\). In the case $k=0$, the operators $\mathcal{T}_i$ reduce to the usual partial derivatives.
For any $y\in\mathbb{R}^d$, the system
\[
\begin{cases}
	\mathcal{T}_i u(x,y) = y_i u(x,y), \quad 1\le i\le d,\\
	u(0,y)=1,
\end{cases}
\]
admits a unique analytic solution on $\mathbb{R}^d$, denoted by $E(x,y)$ and called the \emph{Dunkl kernel}. Moreover, $E(x,y)$ extends holomorphically to $\mathbb{C}^d\times\mathbb{C}^d$. We recall the following basic properties of the Dunkl kernel (see \cite{MR3,ST}).
\begin{lemma}\label{EP4}
	For $x,y\in\mathbb{R}^d$:
	\begin{enumerate}
		\item $E(x,y)=E(y,x)$;
		\item $E(tx,y)=E(x,ty)$ for all $t\in\mathbb{C}$.
	\end{enumerate}
\end{lemma}

\subsection{Dunkl transform}
The Dunkl transform of a function $f\in L^1_k(\mathbb{R}^d)$ is defined by
\[
D_k(f)(y) := c_h \int_{\mathbb{R}^d} f(x) E(x,-iy) \, d\mu_k(x),
\]
where $c_h^{-1} = \int_{\mathbb{R}^d} h_k^2(x) e^{-|x|^2/2}dx$.
In the case $k=0$, $D_k$ reduces to the classical Fourier transform. From \cite{MFE2,CFD3}, the Dunkl transform has the following known properties:
\begin{lemma}\label{Dunkl-transform-properties}
	\begin{enumerate}
		\item If $f\in L^1_k(\mathbb{R}^d)$, then $D_k(f)\in C_0(\mathbb{R}^d)$.
		
		\item (Inversion formula) If both $f$ and $D_k(f)$ belong to $L^1_k(\mathbb{R}^d)$, then
		\[
		f(x)=c_h\int_{\mathbb{R}^d}E(ix,y)D_k(f)(y)\,d\mu_k(y).
		\]
		
		\item The Dunkl transform extends uniquely to an isometric isomorphism on $L^2_k(\mathbb{R}^d)$.
		
	
		\item (Plancherel's formula) For $f\in L^2_k(\mathbb{R}^d)$,
		\[
		\|f\|_{L^2_k}
		=
		\|D_k(f)\|_{L^2_k}.
		\]
	\end{enumerate}
\end{lemma}

\subsection{Dunkl translation}
Following \cite{ST}, we introduce the generalized translation operator.
\begin{definition}
	For $y\in\mathbb{R}^d$, the generalized translation operator
	$f\mapsto \tau_y f$ is defined on $L^2_k(\mathbb{R}^d)$ by
	\[
	D_k(\tau_y f)(x)
	=
	E(y,-ix)\,D_k(f)(x),
	\qquad x\in\mathbb{R}^d.
	\]
\end{definition}
Let $A_k(\mathbb{R}^d) := \{ f\in L^1_k(\mathbb{R}^d) : D_k(f)\in L^1_k(\mathbb{R}^d) \}$.
Then, for every $f\in A_k(\mathbb{R}^d)$, the generalized translation admits the integral representation
\[
\tau_y f(x)
=
\int_{\mathbb{R}^d}
E(ix,\xi)E(-iy,\xi)D_k(f)(\xi)\,d\mu_k(\xi).
\]

We next recall an integral representation of the Dunkl translation operator (see \cite[Theorem 5.1]{MR3}). To this end, we introduce some auxiliary notation. Let $M$ denote the space of Borel measures on $\mathbb{R}^d$, and define
\[
M_{\mathrm{rad}} := \{ \mu \in M : \mu \circ A = \mu \text{ for all } A \in O(N;\mathbb{R}) \}
\]
the subspace of radial measures in $M$. Moreover, let $M^1(\mathbb{R}^d)$ denote the space of probability measures on $\mathbb{R}^d$.
\begin{lemma}\label{LPYZ}
	For each $x,y\in\mathbb{R}^d$ there exists a unique compactly supported radial probability measure $\rho_{x,y}^k\in M^1_{\mathrm{rad}}(\mathbb{R}^d)$ such that for all $f\in C^\infty(\mathbb{R}^d)$,
	\begin{equation}\label{eq:5.1}
		\tau_y f(x) = \int_{\mathbb{R}^d} f \, d\rho_{x,y}^k.
	\end{equation}
	The support of $\rho_{x,y}^k$ is contained in
	\[
	\bigl\{\xi\in\mathbb{R}^d : \min_{g\in G}|x+gy|\le |\xi|\le \max_{g\in G}|x+gy|\bigr\}.
	\]
\end{lemma}
We will need the following result on the support of the Dunkl translation of a compactly supported function.
\begin{lemma}\cite[Theorem~1.7]{JDA1}\label{LJJ}
	Let $f\in L^2_k(\mathbb{R}^d)$, $\operatorname{supp}f\subset B(0,r)$, and $x\in\mathbb{R}^d$. Then
	\[
	\operatorname{supp}\tau_x f \subset \mathcal{O}(B(x,r)) := \bigcup_{\sigma\in G} B(\sigma(x), r).
	\]
\end{lemma}
We recall several elementary properties of the Dunkl translation operator $\tau_y$ (see \cite{ST}), which will be used throughout the paper.
We list them below for convenience.
\begin{lemma}\label{lemma1}
	Let $f\in A_k(\mathbb{R}^d)$ and $g\in L^1_k(\mathbb{R}^d)$ be bounded. Then
	\[
	\int_{\mathbb{R}^d} \tau_y f(\xi) g(\xi) \, d\mu_k(\xi) = \int_{\mathbb{R}^d} \tau_{-y} g(\xi) f(\xi) \, d\mu_k(\xi),
	\]
	and
	\begin{equation}\label{eq1}
		\tau_y f(x) = \tau_{-x} f(-y).
	\end{equation}
\end{lemma}

\begin{lemma}\label{bound}
	The generalized translation operator $\tau_y$, initially defined on $L^1_k(\mathbb{R}^d)\cap L^\infty_k(\mathbb{R}^d)$, can be extended to all radial functions in $L^p_k(\mathbb{R}^d)$ ($1<p<2$); and $\tau_y: L_{\text{rad}}^p(\mathbb{R}^d,d\mu_k) \to L^p_k(\mathbb{R}^d)$ is bounded with operator norm equal to $1$.
\end{lemma}

\begin{lemma}\label{lemma2.1}
	Let $f\in A_k(\mathbb{R}^d)$ be radial and nonnegative. Then $\tau_y f\ge 0$, $\tau_y f\in L^1_k(\mathbb{R}^d)$, and
	\[
	\int_{\mathbb{R}^d} \tau_y f(x) \, d\mu_k(x) = \int_{\mathbb{R}^d} f(x) \, d\mu_k(x).
	\]
\end{lemma}

\subsection{Dunkl convolution}
Using the Dunkl translation operator, we define the Dunkl convolution of functions as follows (see \cite{ST}).
\begin{definition}
	For $f,g\in L^2_k(\mathbb{R}^d)$,
	\[
	f *_k g(x) := \int_{\mathbb{R}^d} f(y) \tau_x g^{\vee}(y) \, d\mu_k(y), \qquad g^{\vee}(y)=g(-y).
	\]
\end{definition}
An equivalent expression is
\[
f *_k g(x) = \int_{\mathbb{R}^d} D_k(f)(\xi) D_k(g)(\xi) E(ix,\xi) \, d\mu_k(\xi).
\]
The following lemma summarizes the basic properties of the Dunkl convolution.
\begin{lemma}\label{LJD1}
	For $f,g\in L^2_k(\mathbb{R}^d)$:
	\begin{enumerate}
		\item $D_k(f *_k g) = D_k(f) D_k(g)$;
		\item $f *_k g = g *_k f$.
	\end{enumerate}
\end{lemma}
\section{Strong annihilating pairs for the Dunkl transform}\label{sec3}
In this section, we establish the strong annihilating pair property for the Dunkl transform associated with $(\varepsilon,k)$-thin sets. 
The argument relies on a Littlewood--Paley decomposition, suitable kernel estimates, and the geometric properties of thin sets. 
We begin with several auxiliary lemmas that will be used throughout the proof.

We consider a pair of orthogonal projections on $L^{2}_{k}(\mathbb{R}^{d})$ defined by:
\[
E_S f = \chi_S f, \quad D_{k}(F_\Sigma f) = \chi_\Sigma D_{k}(f),
\]
where $S$ and $\Sigma$ are measurable subsets of $\mathbb{R}^{d}$.
\begin{lemma}\label{equ}
	Let $S$ and $\Sigma$ be measurable subsets of $\mathbb{R}^{d}$. Then, the following assertions are equivalent:
	\begin{enumerate}
		\item[(1)] $\|F_\Sigma E_S\| < 1$;
		\item[(2)] There exists a constant $D(k,d,S, \Sigma)$ such that for all $f \in L^{2}_{k}(\mathbb{R}^{d})$ supported in $S$,
		\[
		\|f\|_{L^{2}_{k}} \leq D(k, d, S, \Sigma) \|F_{\Sigma^c} f\|_{L^{2}_{k}};
		\]
		\item[(3)] $(S, \Sigma)$ is a strong annihilating pair, i.e., there exists a constant $C(k,d,S, \Sigma)$ such that for all $f \in L^{2}_{k}(\mathbb{R}^{d})$,
		\[
		\|f\|_{L^{2}_{k}} \leq C(k,d, S, \Sigma) \left( \|E_{S^c} f\|_{L^{2}_{k}} + \|F_{\Sigma^c} f\|_{L^{2}_{k}} \right).
		\]
	\end{enumerate}
	Moreover, one may take $D(k,d, S, \Sigma) = (1 - \|F_\Sigma E_S\|)^{-1}$ and $C(k,d,S, \Sigma) = 1 + D(k,d, S, \Sigma)$.
\end{lemma}
\begin{proof}
	We just need to apply \cite[p.88]{VB} with $P=E_S, Q=F_\Sigma$. The above conclusions follow directly.
\end{proof}
\begin{lemma}\label{cutoff}
Let \(\ell\in\mathbb{N}\)  There exists a constant \(C>0\) such that for every \(t>0\),
	\[
	|\tau_x\chi_{B_t(0,2^\ell)}(y)|
	\le
	C2^{\ell(2d+2\gamma_k)}\,\mu_k\bigl(B(x,t)\bigr)^{-1},
	\qquad x,y\in\mathbb R^d,
	\]
	where \(\chi_{B_t(0,2^\ell)}(x)=t^{-\mathcal N}\chi_{B(0,2^\ell)}(x/t)\) and \(\mathcal N=d+2\gamma_k\).
\end{lemma}

\begin{proof}
	Let \(\psi\in C_c^\infty(\mathbb{R}^d)\) be a function satisfying  
	\[
	0\leq \psi \leq 1,\qquad 
	\psi(x)=1\;\text{for }|x|\leq 1,\qquad 
	\psi(x)=0\;\text{for }|x|\geq 2.
	\]
	
	For \(\ell\ge 0\) define  
	\[
	g_\ell(x):=\psi(2^{-\ell}x).
	\]
	Then we have the pointwise bounds  
	\[
	\chi_{B(0,2^\ell)}(x)\;\leq\;g_\ell(x)\;\leq\;\chi_{B(0,2^{\ell+1})}(x).
	\]
	Indeed, if \(x\in B(0,2^\ell)\) then \(|2^{-\ell}x|<1\) and hence \(g_\ell(x)=1\); if \(x\notin B(0,2^{\ell+1})\) then \(|2^{-\ell}x|\ge 2\) and \(g_\ell(x)=0\).
	
	For any multi‑index \(\beta\) we have  
	\[
	\partial^\beta g_\ell(x)=2^{-\ell|\beta|}(\partial^\beta\psi)(2^{-\ell}x),
	\]
	so that  
	\[
	|\partial^\beta g_\ell(x)|\le 2^{-\ell|\beta|}\sup_{y\in\mathbb{R}^d}|\partial^\beta\psi(y)|.
	\]
	
	If \(|x|\ge 2^{\ell+1}\) then \(g_\ell(x)=0\), and consequently  
	\[
	|\partial^\beta g_\ell(x)|\le (1+|x|)^{-d-|\beta|-\mathcal N}.
	\]
	
	If \(|x|<2^{\ell+1}\) we estimate  
	\[
	(1+|x|)^{-d-|\beta|-\mathcal N}\ge(1+2^{\ell+1})^{-d-|\beta|-\mathcal N}\ge(2^{\ell+2})^{-d-|\beta|-\mathcal N},
	\]
	and therefore  
	\begin{align*}
		|\partial^\beta g_\ell(x)|
		&\le 2^{-\ell|\beta|}\sup_{y\in\mathbb{R}^d}|\partial^\beta\psi(y)|\\
		&\le 2^{\ell(d+\mathcal N)+2(d+|\beta|+\mathcal N)}\sup_{y\in\mathbb{R}^d}|\partial^\beta\psi(y)|\;(1+|x|)^{-d-|\beta|-\mathcal N}.
	\end{align*}
	
	Now let \(s\) be an even integer greater than $\mathcal N$ and set 
	\[
	C_{d,\mathcal N,s,\ell}:=\sup_{|\beta|\le s}\;
	2^{\ell(d+\mathcal N)+2(d+|\beta|+\mathcal N)}\sup_{y\in\mathbb{R}^d}|\partial^\beta\psi(y)|=2^{\ell(d+\mathcal N)+2(d+s+\mathcal N)}\sup_{|\beta|\le s}\sup_{y\in\mathbb{R}^d}|\partial^\beta\psi(y)|.
	\]
	Define  
	\[
	h^\ell(x):=\frac{g_\ell(x)}{C_{d,\mathcal N,s,\ell}}.
	\]
	Then \(h^\ell\) satisfies the hypotheses of \cite[Theorem~4.1]{AJ} with \(\kappa=0\), and we have  
	\[
	\chi_{B(0,2^\ell)}(x)\;\le\;C_{d,\mathcal N,s,\ell}\,h^\ell(x)\;\le\;\chi_{B(0,2^{\ell+1})}(x).
	\]
	
	Applying \cite[Theorem~4.1]{AJ} to \(h^\ell\) yields  
	\[
	|\tau_x h^\ell_t(y)|\le C\,\mu_k\bigl(B(x,t)\bigr)^{-1},
	\]
	where \(h^\ell_t(x)=t^{-\mathcal N}h^\ell(x/t)\) and the constant \(C>0\) is the one provided by that theorem.  
	
	Finally, recalling that \(\chi_{B_t(0,2^\ell)}(x)=t^{-\mathcal N}\chi_{B(0,2^\ell)}(x/t)\), we obtain  
	\[
	|\tau_x\chi_{B_t(0,2^\ell)}(y)|
	\le C_{d,\mathcal N,s,\ell}\;|\tau_x h^\ell_t(y)|
	\le C\cdot C_{d,\mathcal N,s,\ell}\;\mu_k\bigl(B(x,t)\bigr)^{-1}.
	\]
	This completes the proof.  
\end{proof}
\begin{remark}\label{re1}
	More generally, if $g\in L_k^\infty(\mathbb{R}^d)$ is supported in $B(0,R)$, then
	\[
	|\tau_x g_t(y)|
	\le
	C_R
	\frac{\displaystyle \|g\|_{L^\infty_k(\mathbb{R}^d)}}
	{\mu_k\bigl(B(x,t)\bigr)},
	\]
	where  $C_R>0$ depends  on $R$, $d$, and $k$, and
	\(	g_t(x)=t^{-\mathcal N}g\!\left(\frac{x}{t}\right)\).

	Indeed, writing
	\[
	g=g_{+}-g_{-},
	\]
	where
	\[
	g_{+}=\max\{g,0\},
	\qquad
	g_{-}=\max\{-g,0\},
	\]
	we have
	\[
	0\le g_{\pm}(x)\le \|g\|_{L^\infty_k(\mathbb{R}^d)}\chi_{B(0,R)}(x).
	\]
Applying Lemma \ref{cutoff} to the characteristic function of the support then yields the desired estimate.
	
	Compared with Theorem~4.1 in \cite{AJ}, the above estimate removes the smoothness assumption when the function $g$ is compactly supported.
\end{remark}
We are now in a position to prove the main result of this section.
\begin{theorem}\label{pair}
	If $\varepsilon$ is sufficiently small and $S,\Sigma \subset \mathbb{R}^d$ are $(\varepsilon,k)$--thin sets,
	then $(S,\Sigma)$ is a strong annihilating pair for the Dunkl transform.
	More precisely, there exists $C=C(k, d, S, \Sigma)>0$ such that for every 
	$f\in L^2_k(\mathbb{R}^d)$, one has
\begin{equation}\label{pair1}
	\|f\|_{L^2_k}
	\le C
	\Big(
	\|\chi_{S^c} f\|_{L^2_k}
	+
	\|\chi_{\Sigma^c} D_k f\|_{L^2_k}
	\Big).
\end{equation}
\end{theorem}
\begin{proof}
	The proof follows the general strategy of Theorem 2.1 in \cite{SVW},
	with appropriate modifications in the Dunkl transform setting.\\
\textbf{Step 1. Littlewood–Paley decomposition.} Let $\phi$ be a radial, real-valued Schwartz function such that $0 \le D_k(\phi) \le 1$, $\operatorname{supp} D_k(\phi) \subset B(0,2)$, and $D_k(\phi) \equiv 1$ on $B(0,1)$.
Define for all integers $j \ge 0$,
\[
\phi_j(x) = 2^{j(d+2\gamma_k)}\,\phi(2^j x),
\qquad 
\gamma_k = \sum_{\alpha \in R_+} k_\alpha.\]
Then
\begin{align*}
	\int_{\mathbb{R}^d} |\phi_j(x)|\, d\mu_k(x)
	&= \int_{\mathbb{R}^d} \left|2^{j(d+2\gamma_k)}\phi(2^j x)\right|\, h_k^2(x)\,dx  \\
	&= \int_{\mathbb{R}^d} |\phi(x)|\, d\mu_k(x).
\end{align*}

Moreover,
\begin{align*}
	D_k(\phi_j)(\xi)
	&= c_h\int_{\mathbb{R}^d} E(x,-i\xi)\,\phi_j(x)\,h_k^2(x)\,dx \\
	&=c_h \int_{\mathbb{R}^d} E(x,-i\xi)\,2^{j(d+2\gamma_k)}\phi(2^j x)\,h_k^2(x)\,dx \\
	&= c_h\int_{\mathbb{R}^d} E(2^{-j}t,-i\xi)\,\phi( t)\,h_k^2(t)\,dt \\
	&= c_h\int_{\mathbb{R}^d} E(t,-i(2^{-j}\xi))\,\phi(t)\,h_k^2(t)\,dt \\
	&= D_k(\phi)(2^{-j}\xi).
\end{align*}
The forth quality follows from Lemma \ref{EP4}.

Hence,
\[
\operatorname{supp}D_k(\phi_j) \subset B(0,2^{j+1}),
\qquad 
D_k(\phi_j)\equiv 1 \ \text{on } B(0,2^j).
\]

Let
\[
\psi_0=D_k(\phi),
\qquad 
\psi_j=D_k(\phi_j)-D_k(\phi_{j-1}),\quad j\ge1 .
\]
Then
\[
\sum_{j=0}^{\infty}\psi_j=1,
\]
and
\[
\operatorname{supp}(\psi_j)\subset 
B(0,2^{j+1})\setminus B(0,2^{j-1}), \quad j\ge1.
\]

Define
\[
L_N f=\sum_{j=0}^N \psi_j\cdot(\phi_j *_k f),
\qquad 
T_N f=\sum_{j=0}^N \psi_j\cdot(f-\phi_j *_k f).
\]

Set
\[
A_N(x,y)=\sum_{j=0}^N \psi_j(x)\,\tau_x\phi_j(y),
\]
and
\[
M_N(\xi,\eta)=\sum_{j=0}^N \tau_\xi D_k\psi_j(\eta)\,
\bigl(1-D_k\phi_j(\eta)\bigr).
\]

Then
\begin{align*}
	L_N f(x)
	&= \sum_{j=1}^N \psi_j(x)
	\int_{\mathbb{R}^d} f(y)\,\tau_x\phi_j^\vee(y)\,d\mu_k(y) \\
		&= \sum_{j=1}^N \psi_j(x)
	\int_{\mathbb{R}^d} f(y)\,\tau_x\phi_j(y)\,d\mu_k(y) \\
	&= \int_{\mathbb{R}^d} \sum_{j=1}^N 
	\psi_j(x)\,\tau_x\phi_j(y)\,f(y)\,d\mu_k(y) \\
	&= \int_{\mathbb{R}^d} A_N(x,y)\,f(y)\,d\mu_k(y).
\end{align*}
The second equality follows from the radiality of $\phi_j$.
Similarly, we have
\begin{align*}
	D_k(T_N f)(\xi)
	&= \sum_{j=1}^N 
	D_k\psi_j * D_k\bigl( f -\phi_j * f\bigr)(\xi) \\
	&= \sum_{j=1}^N 
	\int_{\mathbb{R}^d} \tau_\xi (D_k\psi_j)^\vee(\eta)
	\bigl(D_k f - D_k\phi_j \cdot D_k f\bigr)(\eta)
	\,d\mu_k(\eta) \\
	&= \int_{\mathbb{R}^d} \sum_{j=0}^N 
	\tau_\xi D_k\psi_j(\eta)\,
	\bigl(1-D_k\phi_j(\eta)\bigr)\,
	D_k f(\eta)\,d\mu_k(\eta) \\
	&= \int_{\mathbb{R}^d} M_N(\xi,\eta)\,D_k f(\eta)\,d\mu_k(\eta).
\end{align*}
\textbf{Step 2. Estimates for the kernels $A_N$ and $M_N$.}
$A_N$ and $M_N$ satisfy the following properties for a suitable constant $C$ independent of $N$:
\begin{align}
	&\text{(i)}\quad \int_{\mathbb{R}^d} |A_N(x,y)|\,d\mu_k(y) \le C \quad \text{for all } x,\\
	&\text{(ii)}\quad \int_{\mathbb{R}^d} |A_N(x,y)|\,d\mu_k(x) \le C \quad \text{for all } y,\\
	&\text{(iii)}\quad \int_{\mathbb{R}^d} |M_N(\xi,\eta)|\,d\mu_k(\eta) \le C \quad \text{for all } \xi,\\
	&\text{(iv)}\quad \int_{\mathbb{R}^d} |M_N(\xi,\eta)|\,d\mu_k(\xi )\le C \quad \text{for all } \eta.
\end{align}
Furthermore, if $S$ and $\Sigma$ are $(\varepsilon ,k)$–thin sets, then
\begin{enumerate}
	\item[(v)]
	\[
	\int_S |A_N(x,y)|\,d\mu_k(y) \le C\varepsilon,
	\qquad \text{for all } x.
	\]
	
	\item[(vi)]
	\[
	\int_\Sigma |M_N(\xi,\eta)|\,d\mu_k(\xi ) \le C\varepsilon,
	\qquad \text{for all } \eta.
	\]
\end{enumerate}
\textit{Proof of (i).}
Recall that
\[
A_N(x,y)=\sum_{j=0}^N \psi_j(x)\,\tau_x \phi_j(y).
\]
Hence,
\begin{align}\label{A1}
	\int_{\mathbb{R}^d} |A_N(x,y)|\, d\mu_k(y)
	&\le \sum_{j=0}^N 
	|\psi_j(x)|
	\int_{\mathbb{R}^d} |\tau_x \phi_j(y)|\, d\mu_k(y).
\end{align}
For each fixed $x\in\mathbb{R}^d$, there are at most three indices
$j$ such that $\psi_j(x)\neq 0$.  By the definition of
$\psi_j$, we have
\[
|\psi_j(x)|\le 1.
\]
Moreover, by Lemma \ref{bound}, we have
\[
\int_{\mathbb{R}^d} |\tau_x\phi_j(y)|\, d\mu_k(y)
\le \|\phi_j\|_{L^1_k}
= \|\phi\|_{L^1_k} .
\]
Combining these estimates, we obtain
\[
\int_{\mathbb{R}^d} |A_N(x,y)|\, d\mu_k(y)
\le 3 \|\phi\|_{L^1_k}\leq C.
\]
\textit{Proof of (ii).}
Using the integral representation of the Dunkl translation and Lemma \ref{EP4},
\begin{align*}
	|\tau_x \phi_j(y)|
	&= \left|
	\int_{\mathbb{R}^d} E(-i x,\xi)E(i y,\xi) D_k(\phi_j)(\xi)  d\mu_k(\xi)
	\right|  \\
	&= \left|
	\int_{\mathbb{R}^d} E(-i x,\xi)E(i y,\xi) D_k(\phi)(2^{-j}\xi)  d\mu_k(\xi)
	\right|  \\
	&=  2^{j(d+2\gamma_k)}\left|
	\int_{\mathbb{R}^d} E(-i x,2^{j}\eta)E(i y,2^{j}\eta) D_k(\phi)(\eta)  d\mu_k(\eta)
	\right|  \\
	&=  2^{j(d+2\gamma_k)}\left|
	\int_{\mathbb{R}^d} E(-i 2^{j}x,\eta)E(i 2^{j} y,\eta) D_k(\phi)(\eta)  d\mu_k(\eta)
	\right|  \\
	&= 2^{j(d+2\gamma_k)}\left|
\tau_{2^{j}x} \phi(2^{j}y)
	\right| . 
\end{align*}
Moreover, by estimate (4.7) in \cite{AJ}, we know that for every $M>0$,
\[
|\tau_x\phi_j(y)|
\le C_M\, 2^{j(d+2\gamma_k)}
\bigl(1+2^j \mathcal{D}(x,y)\bigr)^{-M}
\le C_M\, 2^{j(d+2\gamma_k)}
\bigl(1+2^j d(x,y)\bigr)^{-M},
\]
where $\mathcal{D}(x,y) := \min\{ |x - \sigma(y)| : \sigma \in G \}$ and $
d(x,y):=|x-y|.
$\\
Choose $M=3(d+2\gamma_k)$, we have 
\[
|\tau_x\phi_j(y)|
\le C(d,\gamma_k)\, 2^{j(d+2\gamma_k)}
\bigl(1+2^j d(x,y)\bigr)^{-3(d+2\gamma_k)}.
\]
Fix $y$ and let $\sum^*$ denote the sum over all $j\in\{0,\dots,N\}$ such that $\operatorname{dist}(y,\operatorname{supp}\psi_j)\geq 1$. There are at most three values of $j$ with $\operatorname{dist}(y,\operatorname{supp}\psi_j)<1$. Hence,
\begin{align*}
&\int_{\mathbb{R}^d} |A_N(x,y)|\, d\mu_k(x)
\le 3\int_{\mathbb{R}^d} |\tau_x\phi_j(y)|\, d\mu_k(x)+\int_{\mathbb{R}^d} \sum^*|\psi_j(x)\tau_x\phi_j(y)|\, d\mu_k(x)\\
&\le 3\int_{\mathbb{R}^d} |\tau_{-y}\phi_j(-x)|\, d\mu_k(x)+C(d,\gamma_k)\int_{\mathbb{R}^d} \sum^* |\psi_j(x)|\, 2^{j(d+2\gamma_k)}
\bigl(1+2^j d(x,y)\bigr)^{-3(d+2\gamma_k)}d\mu_k(x)\\
&\leq 3\|\phi\|_{L^1_k}+C(d,\gamma_k)\sum^*2^{-2j(d+2\gamma_k)}\|\psi_j\|_{L^1_k}\\
&\leq 3\|\phi\|_{L^1_k}+C(d,\gamma_k)\sum^*2^{-j(d+2\gamma_k)}\|\psi_0\|_{L^1_k}\\
&\leq C.
\end{align*}
\textit{Proofs of (iii) and (iv).}
We rewrite the definition of $M_N(\xi,\eta)$ as follows:
\begin{align}\label{e1}
M_N(\xi,\eta)
&=\sum_{j=0}^N \tau_\xi D_k\psi_j(\eta)
\bigl(1-D_k\phi_j(\eta)\bigr)\notag \\
&=\sum_{j=0}^N \tau_\xi D_k\psi_j(\eta)
\sum_{i>j}\psi_i(\eta))\notag \\
&=\sum_{i=1}^\infty \psi_i(\eta))\sum_{j=0}^{i\ast}\tau_\xi D_k\psi_j(\eta)
\notag \\
&=\sum_{i=1}^\infty \psi_i(\eta))\tau_\xi\phi_{i\ast}(\eta),
\end{align}
where $i\ast=\min\{i-1,N\}$.
Observe the similarity between (\ref{e1}) and the definition of $A_N(x,y)$.
(iv) can be proved by using the same argument as in the proof of (i). For (iii), we further rewrite (\ref{e1}) as
\begin{align}\label{e2}
	M_N(\xi,\eta)
	&=\sum_{i=1}^N \psi_i(\eta))\tau_\xi\phi_{i-1}(\eta)+\sum_{i>N} \psi_i(\eta))\tau_\xi\phi_{N}(\eta)\notag\\
	&=\sum_{i=1}^N \psi_i(\eta))\tau_\xi\phi_{i-1}(\eta)+\bigl(1-D_k\phi_j(\eta)\bigr)\tau_\xi\phi_{N}(\eta),
\end{align}
Because $|1-D_k\phi_j|\le 1$, we obtain
\[
\int_{\mathbb{R}^d}\lvert(1-D_k\phi_j(\eta))\tau_\xi\phi_{N}(\eta)\rvert d\mu_k(\eta)\leq \int_{\mathbb{R}^d}\lvert\tau_\xi\phi_{N}(\eta)\rvert d\mu_k(\eta)\le \|\phi_N\|_{L^1_k}
= \|\phi\|_{L^1_k}.
\]
On the other hand, the estimate
\[
\int_{\mathbb{R}^d}
\left|
\sum_{i=1}^N
\psi_i(\eta)\tau_\xi\phi_{i-1}(\eta)
\right|
\,d\mu_k(\eta)
\le C
\]
can be obtained by repeating the same argument used in the proof of (ii).
Therefore, property (iii) follows.\\
\textit{Proofs of (v) and (vi).}
We only prove (v), since (vi) follows from the same argument in view of \eqref{e1}.
Fix $x\in\mathbb{R}^d$. Then
\begin{align*}
	\int_S |A_N(x,y)|\,d\mu_k(y)
	&\leq \sum_{j=0}^N |\psi_j(x)|
	\int_S |\tau_x\phi_j(y)|\,d\mu_k(y).
\end{align*}
For each fixed $x\in\mathbb{R}^d$, there are at most three indices
$j$ such that $\psi_j(x)\neq 0$. By the support of \(\psi_j(x)\), we have \(2^{j-1}\leq |x|\leq 2^{j+1}\) Since moreover $|\psi_j(x)|\le 1$, it remains to prove that for any \(2^{j-1}\leq |x|\leq 2^{j+1}\),
\[
\int_S |\tau_x\phi_j(y)|\,d\mu_k(y)
\le C\varepsilon
\]
uniformly in $j$.

Because $\phi$ is a Schwartz function, for every sufficiently large integer $T$ there exists a constant $C_T>0$ such that
\[
|\phi(t)|
\le C_T \sum_{m\ge 0} 2^{-mT}\,\chi_{B(0,2^m)}(t).
\]
Using Lemma \ref{lemma2.1}, we obtain
\begin{align*}
	\int_S |\tau_x\phi_j(y)|\,d\mu_k(y)
	&\le \int_S \tau_x|\phi_j|(y)\,d\mu_k(y) \\
	&\le C_T \sum_{m\ge 0} 2^{-mT}
	\int_S (2^j)^{d+2\gamma_k}\,
	\bigl|\tau_x\chi_{B(0,2^{m-j})}(y)\bigr|\,d\mu_k(y).
\end{align*}

We choose a sequence $\{x_k\}$, and cover the set $\operatorname{supp}\tau_x\chi_{B(0,\,2^{m-j})}  $ by a family of balls $B(x_k,\rho(x_k))$.
We first consider the case \(j\le 3\). Since \((2^j)^{d+2\gamma_k}\) is uniformly bounded in this range,
the proof proceeds exactly as in Case 1 below, and the corresponding factor may be absorbed into the constant. Hence, it remains to treat the case \(j>3\). For \(j>3\), we distinguish two cases according to the relative size of \(m\) and \(j\).\\
\textbf{Case 1: $j \leq m$.}
Then $2^{m-j}\ge 1$ and we proceed as follows.
\begin{align}
	&\int_S (2^j)^{d+2\gamma_k}2^{-mT}
	\bigl|\tau_x\chi_{B(0,\,2^{m-j})}(y)\bigr|\,d\mu_k(y)\notag\\
	&\qquad =
	\sum_{x_k}
	\int_{S\cap B(x_k,\rho(x_k))}
	(2^j)^{d+2\gamma_k}2^{-mT}
	\bigl|\tau_x\chi_{B(0,\,2^{m-j})}(y)\bigr|
	\,d\mu_k(y)\notag\\
	&\qquad \leq
2^{-m(T-d-2\gamma_k)}\sum_{x_k}
\int_{S\cap B(x_k,\rho(x_k))}
\bigl|\tau_x\chi_{B(0,\,2^{m-j})}(y)\bigr|
\,d\mu_k(y).
\end{align}
By Lemma~\ref{LJJ}, we have
\[
\operatorname{supp}\tau_x\chi_{B(0,\,2^{m-j})} 
\subset
\bigcup_{\sigma\in G} B(\sigma(x),2^{m-j}).
\]
Since $|x_k|\le |x|+2^{m-j}$ and $|x|\le 2^{j+1}$, we have
\begin{equation}\label{es2}
	\rho(x_k)
	\ge
	\frac{1}{|x|+2^{m-j}}
	\ge
	\frac{1}{2^{j+1}+2^{m-j}}
	\ge
	\frac{1}{2^{m+1}+2^{m-j}}
	\ge
	2^{-m-2}.
\end{equation}

Consequently, the number of balls of radius $\rho(x_k)$ needed to cover
\(
B(\sigma(x),2^{m-j})
\)
is bounded by
\[
C\frac{(2^{m-j})^d}{\rho(x_k)^d}
\le
C\frac{(2^{m-j})^d}{(2^{-m-2})^d}
\le
C2^{2md+2d}.
\]
Moreover, by the definition of $\rho(x_k)$, we have
\(\rho(x_k)\leq 1.\)
Using Theorem 5.1 in \cite{MR3}, the estimate on the number of relevant balls, and the $(\varepsilon,k)$-thinness of $S$, we deduce
\begin{align}
	&\int_S (2^j)^{d+2\gamma_k}2^{-mT}
	\bigl|\tau_x\chi_{B(0,\,2^{m-j})}(y)\bigr|
	\,d\mu_k(y)\notag\\
	&\qquad \leq
	2^{-m(T-d-2\gamma_k)}\sum_{x_k}
	\int_{S\cap B(x_k,\rho(x_k))}
	\bigl|\tau_x\chi_{B(0,\,2^{m-j})}(y)\bigr|
	\,d\mu_k(y)\notag\\
	&\qquad \leq
	2^{-m(T-d-2\gamma_k)}
	\sum_{x_k}
	\mu_k(S\cap B(x_k,\rho(x_k)))\notag\\
	&\qquad \leq
	C2^{-m(T-d-2\gamma_k)}2^{2md}
	\sup_{x_k}
	\mu_k(S\cap B(x_k,\rho(x_k)))\notag\\
	&\qquad \leq
	C\varepsilon
	2^{-m(T-d-2\gamma_k)}2^{2md}
	\sup_{x_k}
	\mu_k(B(x_k,\rho(x_k))).
\end{align}
Next, using the estimate
\[
\mu_k(B(x_k,\rho(x_k)))
\leq
C\,\rho(x_k)^d
\prod_{\alpha\in R_+}
\bigl(|\langle x_k,\alpha\rangle|+\rho(x_k)\bigr)^{2k_\alpha},
\]
which follows from \eqref{vs}, 
together with \eqref{es2} and the fact that $\rho(x_k)\le1$, we obtain
\begin{align}\label{ee1}
	&\int_S (2^j)^{d+2\gamma_k}2^{-mT}
	\bigl|\tau_x\chi_{B(0,\,2^{m-j})}(y)\bigr|
	\,d\mu_k(y)\notag\\
	&\qquad \leq
	C\varepsilon 2^{-m(T-d-2\gamma_k)}2^{2md}\sup_{x_k} \rho(x_k)^d
	\displaystyle\prod_{\alpha\in R_+}
	\bigl(|\langle x_k,\alpha\rangle|+\rho(x_k)\bigr)^{2k_\alpha}\notag\\
	&\qquad \leq
	C\varepsilon
	2^{-m(T-d-2\gamma_k)}2^{2md}
	\sup_{x_k}
	\prod_{\alpha\in R_+}
	\bigl(|x_k||\alpha|+1\bigr)^{2k_\alpha}\notag\\
	&\qquad \leq
	C\varepsilon
	2^{-m(T-d-2\gamma_k)}
	2^{2md}
	(2^{m+3})^{2\gamma_k}\notag\\
	&\qquad \leq
	C\varepsilon
	2^{-m(T-3d-4\gamma_k)}.
\end{align}
\textbf{Case 2: $m<j$.}
By Lemma \ref{lemma1} and Lemma \ref{cutoff}, we have
\begin{equation}\label{kernel-est}
	\bigl|(2^j)^{d+2\gamma_k}\tau_x\chi_{B(0,\,2^{m-j})}(y)\bigr|=	\bigl|(2^j)^{d+2\gamma_k}\tau_y\chi_{B(0,\,2^{m-j})}(x)\bigr|
	\leq C2^{m(2d+2\gamma_k)}\,\mu_k\bigl(B(y,2^{-j})\bigr)^{-1}.
\end{equation}
Hence,
\begin{align}
	&\int_S (2^j)^{d+2\gamma_k}2^{-mT}
	\bigl|\tau_x\chi_{B(0,\,2^{m-j})}(y)\bigr|\,d\mu_k(y)\notag\\
	&\qquad =
	\sum_{x_k}
	\int_{S\cap B(x_k,\rho(x_k))}
	(2^j)^{d+2\gamma_k}2^{-mT}
	\bigl|\tau_x\chi_{B(0,\,2^{m-j})}(y)\bigr|
	\,d\mu_k(y)\notag\\
	&\qquad \leq
	C\sum_{x_k}
	\int_{S\cap B(x_k,\rho(x_k))}
	2^{-mT}2^{m(2d+2\gamma_k)}
	\sup_{y_0\in S\cap B(x_k,\rho(x_k))}
	\mu_k\bigl(B(y_0,2^{-j})\bigr)^{-1}
	\,d\mu_k(y)\notag\\
	&\qquad \leq
	C\sum_{x_k}
	\sup_{y_0\in S\cap B(x_k,\rho(x_k))}
	2^{-m(T-2d-2\gamma_k)}
	\frac{\mu_k(S\cap B(x_k,\rho(x_k)))}
	{\mu_k(B(y_0,2^{-j}))}\notag\\
	&\qquad \leq
	C\varepsilon
	\sum_{x_k}
	\sup_{y_0\in B(x_k,\rho(x_k))}
2^{-m(T-2d-2\gamma_k)}
	\frac{\mu_k(B(x_k,\rho(x_k)))}
	{\mu_k(B(y_0,2^{-j}))}.
	\label{thin-est}
\end{align}
The last inequality follows from the $(\varepsilon,k)$-thinness of $S$.
Using \eqref{vs}, we obtain
\begin{align}
	\frac{\mu_k(B(x_k,\rho(x_k)))}
	{\mu_k(B(y_0,2^{-j}))}
	&\leq
	C\,
	\frac{
		\rho(x_k)^d
		\displaystyle\prod_{\alpha\in R_+}
		\bigl(|\langle x_k,\alpha\rangle|+\rho(x_k)\bigr)^{2k_\alpha}
	}
	{
		(2^{-j})^d
		\displaystyle\prod_{\alpha\in R_+}
		\bigl(|\langle y_0,\alpha\rangle|+2^{-j}\bigr)^{2k_\alpha}
	}.
	\label{measure-ratio}
\end{align}

Moreover, by Lemma~\ref{LJJ},
\[
\operatorname{supp}\tau_x\chi_{B(0,\,2^{m-j})}  
\subset
\bigcup_{\sigma\in G} B(\sigma(x),2^{m-j}),
\]
and therefore the number of relevant balls is bounded by
\[
C\frac{(2^{m-j})^d}{\rho(x_k)^d}.
\]
Since, by the definition of a thin set, 
\begin{equation}\label{es1}
\rho(x_k)\leq\frac{1}{|x|-2^{m-j}}\leq \frac{1}{|x|-1}\leq \frac{1}{2^{j-1}-1}\leq 4 \cdot 2^{-j},
\end{equation}
we deduce from \eqref{thin-est} and \eqref{measure-ratio} that
\begin{align}
	&\int_S (2^j)^{d+2\gamma_k}2^{-mT}
	\bigl|\tau_x\chi_{B(0,\,2^{m-j})}(y)\bigr|
	\,d\mu_k(y)\notag\\
	&\qquad \leq
	C\varepsilon\,2^{md} 2^{-m(T-2d-2\gamma_k)}
	\sup_{y_0\in B(x_k,\rho(x_k))}
	\prod_{\alpha\in R_+}
	\left(
	\frac{
		4+|\langle 2^j x_k,\alpha\rangle|
	}{
		1+|\langle 2^j y_0,\alpha\rangle|
	}
	\right)^{2k_\alpha}.
	\label{product-est}
\end{align}

Now write
\[
2^j y_0 = 2^j x_k +2^j \rho(x_k) x_0,
\qquad x_0\in B(0,1),
\]
and define
\[
I_\alpha
:=
\frac{
	4+|\langle 2^j x_k,\alpha\rangle|
}{
	1+|\langle 2^j x_k,\alpha\rangle+\langle 2^j \rho(x_k)x_0,\alpha\rangle|
}.
\]
Then
\[
\frac{
	4+|\langle 2^j x_k,\alpha\rangle|
}{
	1+|\langle 2^j y_0,\alpha\rangle|
}
= I_\alpha.
\]

If
\[
|\langle 2^j x_k,\alpha\rangle|\leq 10,
\]
then
\[
I_\alpha
\leq
\frac{4+10}{1}
=14.
\]
On the other hand, if
\[
|\langle 2^j x_k,\alpha\rangle|>10,
\]
then, since $x_0\in B(0,1)$, $|\alpha|=\sqrt{2}$, and by \eqref{es1}, we obtain
\begin{align*}
	I_\alpha
	&\leq
	\frac{
		4+|\langle 2^j x_k,\alpha\rangle|
	}{
		1+|\langle 2^j x_k,\alpha\rangle|
		-
		|\langle 2^j \rho(x_k)x_0,\alpha\rangle|
	}\\
	&\leq
	1+
	\frac{
		3+\langle 2^j \rho(x_k)x_0,\alpha\rangle
	}{
		1+|\langle 2^j x_k,\alpha\rangle|-8
	}\\
	&\leq
	1+
	\frac{
		3+8
	}{
		1+|\langle 2^j x_k,\alpha\rangle|-8
	}
	\leq 5.
\end{align*}
Thus,
\[
I_\alpha\leq C,
\qquad \forall \alpha\in R_+,
\]
which implies that
\[
\prod_{\alpha\in R_+} I_\alpha^{2k_\alpha}\leq C.
\]
Substituting this estimate into \eqref{product-est}, we conclude that
\begin{equation}\label{ee2}
\int_S (2^j)^{d+2\gamma_k}2^{-mT}
\bigl|\tau_x\chi_{B(0,\,2^{m-j})}(y)\bigr|
\,d\mu_k(y)
\leq
C\varepsilon 2^{-m(T-3d-2\gamma_k)}.
\end{equation}
Consequently, for \(j>3\), it follows from \eqref{ee1} and \eqref{ee2} that
\[
\int_S |\tau_x\phi_j(y)|\,d\mu_k(y)
\le C\varepsilon \sum_{m<j} 2^{-m(T-3d-2\gamma_k)}
+ C\varepsilon \sum_{m\ge j} 2^{-m(T-3d-4\gamma_k)}
\le C\varepsilon,
\]
provided that \(T>3d+4\gamma_k\).

Hence,
\[
\int_S |A_N(x,y)|\,d\mu_k(y)\le C\varepsilon.
\]
This completes the proof of (v).\\
\textbf{Step 3. Final estimates and the strong annihilating inequality.}
Define
\[
L f=\lim_{N\to\infty}L_N f,
\qquad
T f=\lim_{N\to\infty}T_N f.
\]
If $f\in\mathcal{S}$, then $L_N f$ and $T_N f$ converge in the topology of $\mathcal{S}$.
By construction,
\[
L+T=I,
\]
where $I$ denotes the identity operator.
Moreover, both $L$ and $T$ extend to bounded operators on $L_k^2(\mathbb{R}^d)$.

Using (ii) and (v) (respectively, (iii) and (iv)) together with Schur's test, we obtain
\begin{equation}\label{equa1}
	\|L(\chi_S f)\|_{L_k^2}
	\le C\,\varepsilon^{1/2}\|f\|_{L_k^2},
\end{equation}
and
\begin{equation}\label{equa2}
	\|\chi_\Sigma D_k(T(\chi_S f))\|_{L_k^2}
	\le C\,\varepsilon^{1/2}\|f\|_{L_k^2}.
\end{equation}

Now let
\[
Hf:=\chi_\Sigma D_k(\chi_S f).
\]
Since for $f\in L_k^2(\mathbb{R}^d)$,
\[
\chi_\Sigma D_k(\chi_S f)
=
\chi_\Sigma D_k\bigl(L(\chi_S f)\bigr)
+
\chi_\Sigma D_k\bigl(T(\chi_S f)\bigr),
\]
it follows from \eqref{equa1} and \eqref{equa2} that
\begin{align*}
	\|Hf\|_{L_k^2}
	&\leq
	\|\chi_\Sigma D_k(L(\chi_S f))\|_{L_k^2}
	+
	\|\chi_\Sigma D_k(T(\chi_S f))\|_{L_k^2}\\
	&\leq
	\|L(\chi_S f)\|_{L_k^2}
	+
	C\varepsilon^{1/2}\|\chi_S f\|_{L_k^2}\\
	&\leq
	C\varepsilon^{1/2}\|f\|_{L_k^2}.
\end{align*}
Hence,
\[
\|H\|\leq C\varepsilon^{1/2}.
\]

Taking $\varepsilon>0$ sufficiently small, we obtain
\(\|H\|<1.\)
By Lemma \ref{equ}, we have
	\[
\|f\|_{L^2_k}
\le C
\Big(
\|\chi_{S^c} f\|_{L^2_k}
+
\|\chi_{\Sigma^c} D_k f\|_{L^2_k}
\Big).
\]
Therefore, \((S,\Sigma)\) is a strong annihilating pair for the Dunkl transform.
\end{proof}
\section{Proof of Theorem \ref{T1}}\label{sec4}
In this section, we will prove Theorem \ref{T1}. We begin by establishing, in Lemma \ref{L5}, the equivalence between the uncertainty principle and Unique continuation inequality at two distinct points in time. We then complete the proof of Theorem \ref{T1}.
\begin{lemma}\label{L5}
	Let $A$ and $B$ be measurable subsets of $\mathbb{R}^{d}$. Then the following statements are equivalent:\\
	(i) There exists a positive constant $C_1(k,d,A,B)$ such that for each $f\in L^{2}_k(\mathbb{R}^{d})$,
	\begin{equation}\label{L5.1}
		\int_{\mathbb{R}^{d}}|f(x)|^{2}d\mu_k(x)\leq C_1(k,d,A,B)\left( \int_A |f(x)|^{2}d\mu_k(x)+\int_B |D_k(f)(x)|^{2}d\mu_k(x)\right) .
	\end{equation}  
	(ii) There exists a positive constant $C_2(k,d,A,B)$ such that for each $T>0$ and each $u_0\in L^{2}_k(\mathbb{R}^{d})$,
	\begin{equation}\label{L5.2}
		\int_{\mathbb{R}^{d}}|u_0(x)|^{2}d\mu_k(x)\leq C_2(k,d,A,B)\left( \int_A |u_0 (x)|^{2}d\mu_k(x)+\int_{2TB} |u(x,T;u_0)|^{2}d\mu_k(x)\right) .
	\end{equation}
	Furthermore, when one of the above two statements holds, the constants $C_1(k,d,A,B)$ and $C_2(k,d,A,B)$ can be chosen to be the same number.
\end{lemma}
\begin{proof}
	The proof of this lemma is the same as the proof of Lemma 2.3 in \cite{WWZ}. We just mention that one  needs to replace formula (2.6) in \cite{WWZ} with formula (\ref{key}) in this paper.
\end{proof}
We are now in a position to prove Theorem \ref{T1}.
\begin{proof}
	[\textbf{Proof of Theorem \ref{T1}.}]  Let $T>S\geq 0$ and let $A,B\subset \mathbb{R}^d$ be $(\varepsilon ,k)$–thin sets. By Theorem \ref{pair}, We have (\ref{L5.1}) with $(A,B)$ replaced by $\left( A^{c},B^{c} \right) $ and $C_1(k,d,A,B)$ replaced by $C\left( k,d,A,B \right) $, where  $C\left( k,d,A,B\right) $ is given in (\ref{pair1}). Thus, we can apply Lemma \ref{L5} to get (\ref{L5.2}) with $(A,B)$ replaced by $\left( A^{c},B^{c} \right) $ and $C_2(k,d,A,B)$ replaced by $C\left( k,d,A,B\right) $. So we have
	\begin{equation}\label{T1.2}
		\int_{\mathbb{R}^{d}}|u_0(x)|^{2}d\mu_k(x)\leq C\left( k,d,A,B \right) \left( \int_{A^{c}}|u_0(x)|^{2}d\mu_k(x)+\int_{(2(T-S)B)^{c}}|u(x,T-S;u_0)|^{2}dx\mu_k(x)\right) .
	\end{equation}
	
	Finally, by (\ref{T1.2}), we get
	\begin{equation*}
		\int_{\mathbb{R}^{d}}|u(x,S;u_0)|^{2}d\mu_k(x)\leq C\left( k,d,A,B\right) \left( \int_{A^{c}}|u(x,S;u_0)|^{2}d\mu_k(x)+\int_{(2(T-S)B)^{c}}|u(x,T;u_0)|^{2}d\mu_k(x)\right) .
	\end{equation*}
	By the conservation law for the Schr\"{o}dinger equation, we get the inequality (\ref{T1.1}) in Theorem \ref{T1}.
\end{proof}
\section*{Acknowledgments}
Zhiwen Duan was supported by the National Natural Science Foundation of China under grant 12171178. The authors thank the anonymous referees and the associate editor for their invaluable comments, which helped to improve the paper.
\section*{Data Availability}
Data sharing is not applicable to this article as no new data were created or analyzed in this study.
 
\end{document}